\documentclass[10pt]{amsart}
\usepackage{amssymb}

\newtheorem{theorem}{Theorem}[section]
\newtheorem{proposition}[theorem]{Proposition}
\newtheorem{lemma}[theorem]{Lemma}

\newtheorem{cor}[theorem]{Corollary}
\theoremstyle{definition}
\newtheorem{definition}[theorem]{Definition}

\newtheorem{remark}[theorem]{Remark}

\numberwithin{equation}{section}

\begin{document}
\title[Zeros of Hypergeometric functions]
{Zeros of hypergeometric functions in the $p$-adic setting}


{}
\author{Neelam Saikia}
\address{Department of Mathematics, Indian Institute of Science, Bangalore, INDIA}
\curraddr{}
\email{neelamsaikia@iisc.ac.in/ nlmsaikia1@gmail.com}
\thanks{}


\thanks{The author acknowledges the financial support of Department of Science and Technology, Government of India for financial support under INSPIRE Faculty Award.}

\begin{abstract} 
Let $p$ be an odd prime and $\mathbb{F}_p$ be the finite field with $p$ elements. McCarthy \cite{mccarthy-pacific} initiated a study of hypergeometric functions in the $p$-adic setting. This function can be understood as $p$-adic analogue of Gauss' hypergeometric function, and also some kind of extension of Greene's hypergeometric function over $\mathbb{F}_p$. In this paper we investigate values of two generic families of McCarthy's hypergeometric functions denoted by ${_nG_n}(t)$, and ${_n\widetilde{G}_n}(t)$ for $n\geq3$, and $t\in\mathbb{F}_p$. The values of the function ${_nG_n}(t)$ certainly depend on whether $t$ is $n$-th power residue modulo $p$ or not. Similarly, the values of the function ${_n\widetilde{G}_n}(t)$ rely on the incongruent modulo $p$ solutions of $y^n-y^{n-1}+\frac{(n-1)^{n-1}t}{n^n}\equiv0\pmod{p}$.
These results generalize special cases of $p$-adic analogues of Whipple's theorem and Dixon's theorem of classical hypergeometric series. We examine zeros of the functions ${_nG_n}(t)$, and ${_n\widetilde{G}_n}(t)$ over $\mathbb{F}_p$. Moreover, we look into the values of $t$ for which 
${_nG_n}(t)=0$ for infinitely many primes. For example, we show that there are infinitely many primes for which 
${_{2k}G_{2k}}(-1)=0$. In contrast, for $t\neq0$ there is no prime for which ${_{2k}\widetilde{G}_{2k}}(t)=0$.
\end{abstract}
\maketitle
\section{Introduction}
McCarthy \cite{mccarthy-pacific} introduced a function in terms of quotients of $p$-adic gamma functions that can be understood as analogue of classical hypergeometric series in the $p$-adic setting. He developed this function to generalize Greene's `so called' Gaussian hypergeometric functions to wider classes of primes. In this paper we focus in the investigation of the values of certain families of McCarthy's hypergeometric functions in the $p$-adic setting. To be specific, for a positive integer $n\geq3$, we consider two families of McCarthy's hypergeometric functions in the $p$-adic setting with $n$ pairs of arbitrary parameters essentially depend on $n$ to determine all the possible values and explore their number theoretic consequences.
For a complex number $a$ and a non negative integer $k$ the rising factorial denoted by $(a)_k$ is defined by 
$(a)_k:=a(a+1)(a+2)\cdots(a+k-1)$ for $k>0$ and $(a)_0:=1.$ For $a_i,b_i,\lambda\in\mathbb{C}$ with $b_i\not\in\{\ldots,-3,-2,-1,0\},$ the classical hypergeometric series ${_{r+1}F_r}(\lambda)$ is given by
\vspace{-.3cm}
$$
{_{r+1}}F_{r}\left(\begin{array}{cccc}
                   a_1, & a_2, & \ldots, & a_{r+1} \\
                    & b_1, & \ldots, & b_r
                 \end{array}\mid \lambda
\right):=\sum_{k=0}^{\infty}\frac{(a_1)_k\cdots(a_{r+1})_k}
{(b_1)_k\cdots(b_r)_k}\cdot\frac{\lambda^k}{k!}.
$$
In this series if one of the numerator parameters is equal to a non-positive integer, for example let $a_1=-n$, where $n\in\mathbb{N}\cup\{0\}$ then the series terminates and the function is a polynomial of degree $n$ in $\lambda$.
The problem of describing the zeros of the polynomials ${_2F_1}\left(\begin{array}{cc}
         -n, & b \\
         & c
        \end{array}\mid z\right)$ when $b$ and $c$ are complex arbitrary parameters is of particular interest. Indeed it has not been explored completely, and even when $b,c$ are real. 
      The zero location of special classes of such polynomials with restrictions on the parameters $b$ and $c$ can be found in \cite{4, 6, 7, 8, 11, 13} and the asymptotic zero distribution of certain classes have been investigated in \cite{5, 10, 14, 15, 27}. These give motivation to study zeros of $p$-adic analogue of hypergeometric functions from number theoretic point of view. 
         
         Let $\Gamma_p(\cdot)$ denote the Morita's $p$-adic gamma function and $\omega$ denote the Teichm\"{u}ller character of 
$\mathbb{F}_p$ satisfying $\omega(a)\equiv a\pmod{p}$. Let $\overline{\omega}$ denote the character inverse of $\omega$. For $x\in\mathbb{Q}$ let $\lfloor x\rfloor$ denote the greatest integer less than or equal to $x$ and $\langle x\rangle$ denote the fractional part of $x$, satisfying $0\leq\langle x\rangle<1$.
In these notations, we recall McCarthy's definition of hypergeometric function in the $p$-adic setting.
\begin{definition}\cite[Definition 5.1]{mccarthy-pacific} 
Let $p$ be an odd prime and $t \in \mathbb{F}_p$.
For positive integer $n$ and $1\leq k\leq n$, let $a_k$, $b_k$ $\in \mathbb{Q}\cap \mathbb{Z}_p$.
Then 
\begin{align}
&{_n\mathbb{G}_n}\left[\begin{array}{cccc}
             a_1, & a_2, & \ldots, & a_n \\
             b_1, & b_2, & \ldots, & b_n
           \end{array}\mid t
 \right]:=\frac{-1}{p-1}\sum_{a=0}^{p-2}(-1)^{an}~~\overline{\omega}^a(t)\notag\\
&\times \prod\limits_{k=1}^n(-p)^{-\lfloor \langle a_k \rangle-\frac{a}{p-1} \rfloor -\lfloor\langle -b_k \rangle +\frac{a}{p-1}\rfloor}
 \frac{\Gamma_p(\langle a_k-\frac{a}{p-1}\rangle)}{\Gamma_p(\langle a_k \rangle)}
 \frac{\Gamma_p(\langle -b_k+\frac{a}{p-1} \rangle)}{\Gamma_p(\langle -b_k \rangle)}.\notag
\end{align}
\end{definition}
\noindent This function is also known as $p$-adic hypergeometric function. It is important to note that the value of this function depends only on the fractional part of the parameters $a_k$ and $b_k$. Therefore, we may assume that $0\leq a_k,b_k<1$. 
Consider two families of McCarthy's hypergeometric functions in the $p$-adic setting
 \begin{align*}
{_nG_n}(t):&={_n\mathbb{G}_n}\left[\begin{array}{ccccc}
\frac{1}{2n}, & \frac{3}{2n}, & \frac{5}{2n}, & \ldots, & \frac{2n-1}{2n}\vspace{1mm}\\
0, & \frac{1}{n}, & \frac{2}{n}, & \ldots, & \frac{n-1}{n}
\end{array}\mid t\right],
\end{align*}
and
\begin{align*}
{_n\widetilde{G}_n}(t):&={_n\mathbb{G}_n}\left[\begin{array}{ccccc}
 \frac{1}{2}, & \frac{1}{2(n-1)}, & \frac{3}{2(n-1)}, & \ldots, & \frac{2n-3}{2(n-1)} \vspace{1mm}\\
 0, & \frac{1}{n}, & \frac{2}{n}, & \ldots, & \frac{n-1}{n}
 \end{array}\mid t\right].
\end{align*}
In this paper we restrict our attention to these functions. For convenience we denote these functions by ${_nG_n}(t)$, and 
${_n\widetilde{G}_n}(t)$ throughout the paper. Classical hypergeometric series possess many powerful identities. For instance, 
Gauss' theorem provides a special value of a general ${_2F_1}(1)$ hypergeometric series. To be specific, Gauss established that
\begin{align}
{_2F_1}\left(\begin{array}{cc}
           a, & b \\
           & c
         \end{array}\mid1\right)=\frac{\Gamma(c)\Gamma(c-a-b)}{\Gamma(c-a)\Gamma(c-b)},\notag
\end{align}
provided $R(c-a-b)>0$. In the classical case another evaluation of a $_2F_1(-1)$ hypergeometric series is due to Kummer \cite{kummer}. There are other major summation theorems of 
classical hypergeometric series including Dixon's theorem \cite[p. 51]{slater}
\begin{align}\label{Dixon}
{_3F_2}&\left(\begin{array}{ccc}
a, & b, & c\\
~& 1+a-b, & 1+a-c
\end{array}\mid1\right)\notag\\
&=\frac{\Gamma(1+\frac{a}{2})\Gamma(1+a-b)\Gamma(1+a-c)\Gamma(1+\frac{a}{2}-b-c)}
{\Gamma(1+a)\Gamma(1+\frac{a}{2}-b)\Gamma(1+\frac{a}{2}-c)\Gamma(1+a-b-c)},
\end{align}
and Whipple's theorem \cite[p. 54]{lidl}
\begin{align}\label{Whipple}
{_3F_2}&\left(\begin{array}{ccc}
a, & 1-a, & c\\
~& b, & 1-b+2c
\end{array}\mid1\right)\notag\\
&=\frac{\pi\Gamma(b)\Gamma(1-b+2c)}
{\Gamma(\frac{1}{2}(a+b))\Gamma(\frac{1}{2}(1+a-b+2c))\Gamma(\frac{1}{2}(1-a+b))\Gamma(\frac{1}{2}(2-a-b+2c))}.
\end{align}
These identities give motivation to study values of hypergeometric functions in the $p$-adic setting. 
 \par In Section 2 we discuss the values of ${_nG_n}(t)$ over finite fields. By investigating the values we obtain zeros of these functions. In Section 3 we explore the values of  ${_n\widetilde{G}_n}(t)$ and examine their zeros. Most of the results are deduced by simplifying certain character sums and expressing the characters sums as hypergeometric functions in the $p$-adic setting.
Section 4 includes the basic properties of characters, Gauss sums, Jacobi sums, and $p$-adic gamma functions we need. In Section 4 we also state Hasse-Davenport theorem and Gross-Koblitz formula which we use several times. Section 5 is devoted to the proofs of the main results that are stated in Section 2 and Section 3.

\section{Values of ${_nG_n}(t)$}
\noindent Before we discuss the general case, we first mention that the values of ${_2G_2}(t)$ over $\mathbb{F}_p$ have been  obtained in \cite{neelam-pacific}. We now investigate the values of ${_nG_n}(t)$ when $n\geq3$.
\begin{theorem}\label{special-value-1}
Let $n\geq3$ be a positive integer and $p$ be an odd prime such that $p\nmid n$. Let  
$d=\gcd{(n,p-1)}$ and $t\in\mathbb{F}_p$.
\begin{enumerate}
\item If $t^{\frac{p-1}{d}}\equiv1\pmod{p}$, then
${_nG_n}(t)=\displaystyle\sum_{\substack{a\in\mathbb{F}_p\\ a^n\equiv t\pmod{p}}}\varphi(a)\varphi(a-1)$.
\item If $t^{\frac{p-1}{d}}\not\equiv1\pmod{p}$, then ${_nG_n}(t)=0$.
\end{enumerate}
\end{theorem}
By this theorem it follows easily that the function ${_nG_n}(t)$ takes only integer values. 
It is of interest to examine bounds of hypergeometric functions. From Theorem \ref{special-value-1} we obtain an obvious bound of the function ${_nG_n}(t)$. We express this in the following corollary.
\begin{cor}\label{cor-8}
Let $n\geq3$ and $p$ be an odd prime such that $p\nmid n$. If $t\in\mathbb{F}_p$ and $d=\gcd{(n,p-1)}$ then $-d\leq{_nG_n}(t)\leq d$.
\end{cor}
There are two questions that emerge from Theorem \ref{special-value-1}. The first concerns the investigation of zeros of the $p$-adic hypergeometric function ${_nG_n}(t)$ over $\mathbb{F}_p$. If $x\in\mathbb{Q}$ is a zero of ${_nG_n}(t)$ over 
$\mathbb{F}_p$ then the second question is whether $x$ is a zero of  ${_nG_n}(t)$ over $\mathbb{F}_p$ for infinitely many primes $p$. We first discuss a very special case that gives zeros of the function 
${_3G_3}(t)={_3\mathbb{G}_3}\left[\begin{array}{ccc}
\frac{1}{2}, & \frac{1}{6}, & \frac{5}{6}\vspace{1mm}\\
0, & \frac{1}{3}, & \frac{2}{3}
\end{array}\mid t\right]$. The more general case will be discussed later.
\begin{theorem}\label{analogue-1}
Let $p>3$ be a prime and $t\in\mathbb{F}_p^{\times}$.
\begin{enumerate}
\item Let $p\equiv1\pmod{3}$ and $g$ be a primitive root modulo $p$. If $t\neq1$ then 
\begin{align}
{_3G_3}(t)=0\notag
\end{align}
if and only if $t=g^i$ with $\gcd{(i,3)}=1$, and if $t=1$ then ${_3G_3(1)}=0$ if and only if $p\equiv7\pmod{12}$. 
\item Let $p\not\equiv1\pmod{3}$. Then ${_3G_3}(t)=0$ if and only if $t=0,1$. In other words, if $t\neq0,1$ then ${_3G_3}(t)\neq0$. 
\end{enumerate}
\end{theorem}
\begin{remark}
It is important to note that 1 is a zero of the function ${_3G_3}(t)$ for infinitely many primes. If $g\in\mathbb{Z}$ is a primitive root modulo ${p}$ for infinitely many primes of the form $3k+1$ then $g$ is a zero of the function ${_3G_3}(t)$ for infinitely many primes of the form $3k+1$. The existence of such kind of primitive roots certainly related to Artin's conjecture on primitive roots. 
\end{remark}
If we put $a=\frac{1}{2}$, $b=\frac{1}{6}$, and $c=\frac{5}{6}$ in \eqref{Dixon} we obtain
\begin{align}\label{dixon-value}
{_3F_2}&\left(\begin{array}{ccc}
\frac{1}{2}, & \frac{1}{6}, & \frac{5}{6}\vspace{1mm}\\
 & 1+\frac{1}{3}, & \frac{2}{3}
\end{array}\mid1\right)=\frac{\Gamma(1+\frac{1}{4})\Gamma(1+\frac{1}{3})\Gamma(\frac{2}{3})\Gamma(\frac{1}{4})}
{\Gamma(1+\frac{1}{2})\Gamma(1+\frac{1}{12})\Gamma(\frac{5}{12})\Gamma(\frac{1}{2})}.
\end{align}
This is a particular case of Dixon's theorem \eqref{Dixon} and gives motivation to present $p$-adic analogue of 
\eqref{dixon-value} in the following corollary. 
\begin{cor}\label{cor-1}
If $p>3$ is a prime then we have
\begin{align}
{_3G_3}(1)=\left\{\begin{array}{ll}
\pm2, & \hbox{if $p\equiv1\pmod{12}$;}\\
0, & \hbox{if $p\not\equiv1\pmod{12}$.}
\end{array}\right.\notag
\end{align}
\end{cor}
Moreover, if we put $n=3$ in Theorem \ref{special-value-1} then for $t\in\mathbb{F}_p$ we obtain the values of ${_3G_3}(t)$ for all primes $p>3$. To be specific, we obtain
\begin{cor}\label{cor-10}
\begin{enumerate}
\item Let $p\equiv1\pmod{3}$ and $t\neq0,1$. 
If $t^{\frac{p-1}{3}}\equiv1\pmod{p}$ then 
${_3G_3}(t)=\displaystyle\sum_{\substack{a\in\mathbb{F}_p \\a^3\equiv t\pmod{p}}}\varphi(a(a-1)),$ and if $t^{\frac{p-1}{3}}\not\equiv1\pmod{p}$ then ${_3G_3}(t)=0$.
\item Let  $p\equiv2\pmod{3}$ and $t\neq0,1$. Then 
${_3G_3}(t)=\varphi(t^{\frac{2p-1}{3}})\varphi(t^{\frac{2p-1}{3}}-1)$.
\end{enumerate}
\end{cor}
This gives an extension of Corollary \ref{cor-1} over $\mathbb{F}_p$ and can be understood as some kind of extension over 
$\mathbb{F}_p$ of the particular case \eqref{dixon-value} of Dixon's theorem in the $p$-adic setting.
\par We now discuss the case for $n>3$. In Theorem \ref{special-value-1} we show that if $t^{\frac{p-1}{d}}\not\equiv1\pmod{p}$ then ${_nG_n}(t)=0$ for all primes 
$p\nmid n$, where $d=\gcd{(n,p-1)}$. More generally, it is interesting to examine the zeros of ${_nG_n}(t)$ over $\mathbb{F}_p$. 
\begin{cor}\label{cor-2}
Let $n\geq3$ be a positive integer and $p$ be an odd prime such that $p\nmid n$. Let $d=\gcd{(n,p-1)}$. 
Let $t^{\frac{p-1}{d}}\equiv1\pmod{p}$. 
\begin{enumerate}
\item Let $n$ be even. If $t\neq1$  and $a_1,a_2,\ldots,a_d$ are the incongruent solutions of $y^n\equiv t\pmod{p}$ modulo $p$ and $\displaystyle\sum_{i=1}^{d}\varphi(a_i(a_i-1))=0$ then ${_nG_n}(t)=0$. On the other hand if $t=1$ then ${_nG_n}(1)\neq0$.
\item Let $n$ be odd. If $t\neq1$ then ${_nG_n}(t)\neq0$. On the other hand if $t=1$ and $a_1,a_2,\ldots,a_{d-1}$ different from $1$ are the incongruent solutions of the congruence $y^n\equiv 1\pmod{p}$ modulo $p$ such that $
\displaystyle\sum_{i=1}^{d}\varphi(a_i(a_i-1))=0$ then ${_nG_n}(1)=0$.
\end{enumerate}
\end{cor}
The following corollaries provide particular values of $t\in\mathbb{Q}$ such that  ${_nG_n}(t)=0$ for infinitely many primes.
\begin{cor}\label{cor-6}
Let $n\geq3$ be an odd integer. Then ${_nG_n}(1)=0$ for infinitely many primes $p$ such that $\gcd{(n,p(p-1))}=1$. If 
$t\neq0,1$ then ${_nG_n}(t)=\varphi\left(\frac{a}{a-1}\right)\neq0$ for all primes $p$ such that $\gcd{(n,p(p-1))}=1$, where $a$ is the unique solution of $y^n\equiv t\pmod{p}$ modulo $p$.
\end{cor}
\begin{cor}\label{cor-7}
Let $n\geq4$ be an even integer. Then ${_nG_n}(-1)=0$ for infinitely many primes $p$ such that $p\equiv3\pmod{4}$. 
\end{cor}
If we consider the family $\{{_{p-1}G_{p-1}(t)}:p\geq3\}$ then it turns out that for $t\neq1$ this family takes only one value. We state these in the following corollary explicitly.
\begin{cor}\label{cor-3}
If $t\neq1$ then ${_{p-1}G_{p-1}}(t)=0$, and if $t=1$ then 
${_{p-1}G_{p-1}}(1)=-1$ for all primes $p\geq3$.
\end{cor}
\section{Values of ${_n\widetilde{G}_n}(t)$}
In this section we explore the values of the function ${_n\widetilde{G}_n}(t)$. We express these values in terms of roots of certain polynomial over $\mathbb{F}_p$. 
\begin{theorem}\label{special-value-2}
 Let $n\geq3$ be an integer and $p$ be an odd prime such that $p\nmid n(n-1)$. For $t\in\mathbb{F}_p^\times$ let 
$f_t(y)=y^n-y^{n-1}+\frac{(n-1)^{n-1}t}{n^n}$ be a polynomial over $\mathbb{F}_p$, and $$\beta_n(t)=\left\{\begin{array}{ll}
1, &\hbox{if $n$ is odd;}\\
1-(p-1)\varphi((1-n)t), & \hbox{if $n$ is even.}
\end{array}
\right.$$

\noindent Then we have 
\begin{align}
&{_n\widetilde{G}_n}(t)
 =\frac{\beta_n(t)-1}{p}+\sum_{f_t(a)\equiv0\pmod{p}}\varphi(a(a-1)).\notag
\end{align}
\end{theorem}
If we put $a=\frac{1}{4}$, $b=\frac{1}{3}$,  and $c=\frac{1}{2}$ in \eqref{Whipple} then we obtain
\begin{align}\label{whipple-value}
{_3F_2}\left(\begin{array}{ccc}
\frac{1}{4}, & \frac{3}{4}, & \frac{1}{2}\vspace{1mm}\\
& \frac{1}{3}, & 1+\frac{2}{3}
\end{array}\mid1\right)=\frac{\pi \Gamma(\frac{1}{3})\Gamma(\frac{5}{3})}{\Gamma(\frac{7}{24})
\Gamma(\frac{23}{24})\Gamma(\frac{13}{24})\Gamma(\frac{29}{24})}.
\end{align}
This is a special case of Whipple's theorem. Using this identity we motivate our next result.
If we put $n=3$ in Theorem \ref{special-value-2} then we have 
${_3\widetilde{G}_3}(t)={_3\mathbb{G}_3}\left[\begin{array}{ccc}
\frac{1}{4}, & \frac{3}{4}, & \frac{1}{2}\vspace{1mm}\\
0, & \frac{1}{3}, & \frac{2}{3}
\end{array}\mid t\right]$ and we obtain
\begin{cor}\label{SV-2}
Let $p>3$ be a prime and $t\in\mathbb{F}_p^\times$. Let $27y^3-27y^2+4t$ be a polynomial over $\mathbb{F}_p$.
\begin{enumerate}
\item \begin{align}\label{value-111}
{_3\widetilde{G}_3}(1)=1+\varphi(-2)=\left\{\begin{array}{ll}
2, & \hbox{if $p\equiv1,3\pmod{8}$;}\\
0, & \hbox{if $p\equiv5,7\pmod{8}$.}
\end{array}
\right.
\end{align}
\item Let $t\neq1$ and $27y^3-27y^2+4t$ be irreducible over $\mathbb{F}_p$. Then
${_3\widetilde{G}_3}(t)=0$.
\item Let $t\neq1$ and $27y^3-27y^2+4t$ has one root $a\in\mathbb{F}_p$ counting with multiplicity. Then
${_3\widetilde{G}_3}(t)=\varphi(a(a-1))$.
\item Let $t\neq1$ and $a_1,a_2, a_3$ be the roots of the polynomial $27y^3-27y^2+4t$ in $\mathbb{F}_p$. Then
${_3\widetilde{G}_3}(t)=\displaystyle\sum_{i=1}^3\varphi(a_i(a_i-1))$.
\end{enumerate}
\end{cor}
\eqref{value-111} can be described as a $p$-adic analogue of \eqref{whipple-value} which is a special case of Whipple's theorem. Indeed this corollary extends \eqref{whipple-value} over $\mathbb{F}_p$ in the $p$-adic setting. 
\begin{remark}
It is of interest to discuss the zeros of the function ${_n\widetilde{G}_n}(t)$ over $\mathbb{F}_p$. For example, if $t=1$ then 
${_3\widetilde{G}_3}(1)=0$ if and only if $p\equiv5,7\pmod{8}$. Therefore, $t=1$ is a non trivial zero of the function 
${_3\widetilde{G}_3}(t)$ for infinitely many primes $p$. Let $t\in\mathbb{Q}-\{0,1\}$ and $27y^3-27y^2+4t$ be irreducible over 
$\mathbb{F}_p$ for infinitely primes $p$ then ${_3\widetilde{G}_3}(t)=0$ for infinitely many primes $p$. However, we do not know whether such $t$ exist or not. This will be an interesting question to study in future.
\end{remark}
We now investigate two more general cases of zeros of $p$-adic hypergeometric functions. By Theorem \ref{special-value-2} we obtain
\begin{cor}\label{cor-4}
Let $n>3$ be an even integer and $p$ be an odd prime satisfying $p\nmid n(n-1)$. 
If $t\in\mathbb{F}_p$ then ${_n\widetilde{G}_n}(t)=0$
if and only if $t=0$. In other words, if $t\neq0$ then ${_n\widetilde{G}_n}(t)\neq0$ for all primes 
$p\nmid n(n-1)$.
\end{cor}
\begin{remark}This corollary is in contrast to Corollary \ref{cor-7}. If $p\mid n(n-1)$ then the function ${_n\widetilde{G}_n}(t)$ is not defined for these primes. Hence, by Corollary \ref{cor-4} we obtain that if $t\neq0$ then there is no prime $p$ for which 
${_n\widetilde{G}_n}(t)=0$.
\end{remark}
\section{Notation and Preliminaries}
\subsection{Multiplicative characters:}
Let $\widehat{\mathbb{F}_p^\times}$ denote the group of all multiplicative characters of $\mathbb{F}_p^{\times}$. Let $\overline{\chi}$ denote the inverse of a multiplicative character $\chi$. We extend the domain of each $\chi\in\widehat{\mathbb{F}_p^\times}$ to $\mathbb{F}_p$ by simply setting $\chi(0):=0$ including the trivial character $\varepsilon.$ 
We start with a lemma that gives orthogonality relation of multiplicative characters.
\begin{lemma}\cite[Chapter 8]{ireland}
Let $p$ be an odd prime. Then 
\begin{align}\label{orthogonal-1}
\sum_{\chi\in\widehat{\mathbb{F}_p^\times}}\chi(x)=\left\{
   \begin{array}{ll}
    p-1 , & \hbox{if $x=1$;} \\
  0, & \hbox{if $x\neq1$.}
   \end{array}
 \right.
\end{align}
\end{lemma}
Let $\zeta_p$ denote a fixed primitive $p$-th root of unity. For multiplicative character
 $\chi$ of $\mathbb{F}_p^{\times}$ the Gauss sum is defined by
\begin{align}
g(\chi):=\sum\limits_{x\in \mathbb{F}_p}\chi(x)~\zeta_p^x.\notag
\end{align}
If $\chi=\varepsilon$ then it is easy to see that $g(\varepsilon)=-1$. For more details on the properties of Gauss sum, see \cite{berndt}. Let $\delta: \widehat{\mathbb{F}_p^\times}\mapsto\{0,1\}$ be defined by
\begin{align}
\delta(\chi)=\left\{
   \begin{array}{ll}
    1 , & \hbox{if $\chi=\varepsilon$;} \\
  0, & \hbox{if $\chi\neq\varepsilon$.}
   \end{array}
 \right.
\end{align}
We now state a product formula for Gauss sums.
\begin{lemma}\cite[eq. 1.12]{greene} Let $\chi\in \widehat{\mathbb{F}_p^\times}$. Then
\begin{align}\label{inverse}
g(\chi)g(\overline{\chi})=p\cdot\chi(-1)-(p-1)\delta(\chi).
\end{align}
\end{lemma}
For multiplicative characters $\chi$ and $\psi$ of $\mathbb{F}_p$ the Jacobi sum is defined by
\begin{align}\label{jacobi}
J(\chi,\psi):=\sum_{y\in\mathbb{F}_p}\chi(y)\psi(1-y),
\end{align}
and the normalized Jacobi sum known as binomial is defined by
\begin{align}\label{binomial}
{\chi\choose \psi}:=\frac{\psi(-1)}{p}J(\chi,\overline{\psi}).
\end{align}
The following relation provides a relation between Gauss and Jacobi sums.
\begin{lemma}\cite[eq. 1.14]{greene} 
Let $\chi_1,\chi_2\in \widehat{\mathbb{F}_p^\times}$. Then
\begin{align}\label{gauss-jacobi}
J(\chi_1,\chi_2)=\frac{g(\chi_1)g(\chi_2)}{g(\chi_1\chi_2)}+(p-1)\chi_2(-1)\delta(\chi_1\chi_2).
\end{align}
\end{lemma}
The following product formula of Hasse-Davenport is very important.
\begin{theorem}\cite[Hasse-Davenport relation, Theorem 11.3.5]{berndt}
Let $\psi$ be a multiplicative character of $\mathbb{F}_p^\times$ of order $m$ for some positive integer $m$. For a multiplicative character $\chi$ of $\mathbb{F}_p^\times$ we  have
\begin{align}\label{hd}
\prod_{i=0}^{m-1}g(\chi\psi^i)=g(\chi^m)\chi^{-m}(m)\prod_{i=1}^{m-1}g(\psi^i).
\end{align}
\end{theorem}
We now recall certain properties of binomial from \cite[eq. 2.12, eq. 2.7]{greene}.
\begin{align}\label{rel-1}
{\chi\choose\varepsilon}={\chi\choose\chi}=\frac{-1}{p}+\frac{p-1}{p}\delta(\chi),
\end{align}
and
\begin{align}\label{rel-2}
{\chi\choose\psi}={\chi\choose\chi\overline{\psi}}.
\end{align}
\begin{proposition}\label{proposition-root}\cite[Proposition 2.37]{niven}
Let $p$ is a prime and $\gcd{(b,p)}=1$. If $n$ is a positive integer and $d=\gcd{(n,p-1)}$ then the congruence $y^n\equiv b\pmod{p}$ has $d$ solutions or no solution according as $b^{\frac{p-1}{d}}\equiv1\pmod{p}$ or not.
\end{proposition}
\subsection{$p$-adic preliminaries:}
Let $\overline{\mathbb{Q}_p}$ denote the algebraic closure of $\mathbb{Q}_p$ and $\mathbb{C}_p$ denote the completion of $\overline{\mathbb{Q}_p}$. 
For a positive integer $n$, the $p$-adic gamma function $\Gamma_p(n)$ is defined as
\begin{align}
\Gamma_p(n):=(-1)^n\prod\limits_{0<j<n,p\nmid j}j.\notag
\end{align}
 It can be extended to all $x\in\mathbb{Z}_p$ by setting $\Gamma_p(0):=1$ and for $x\neq0$
\begin{align}
\Gamma_p(x):=\lim_{x_n\rightarrow x}\Gamma_p(x_n).\notag
\end{align}
Let $\pi \in \mathbb{C}_p$ be the fixed root of $x^{p-1} + p=0$ and
$\pi \equiv \zeta_p-1 \pmod{(\zeta_p-1)^2}$. The result given below is known as Gross-Koblitz formula. 
\begin{theorem}\cite[Gross-Koblitz formula]{gross}\label{gross-koblitz} For $j\in \mathbb{Z}$,
\begin{align}
g(\overline{\omega}^j)=-\pi^{(p-1)\langle\frac{j}{p-1} \rangle}\Gamma_p\left(\left\langle \frac{j}{p-1} \right\rangle\right).\notag
\end{align}
\end{theorem}
 An important product formula of $p$-adic gamma functions is given below.
If $m\in\mathbb{Z}^+$, 
$p\nmid m$,  and $x=\frac{r}{p-1}$ with $0\leq r\leq p-1$ then
\begin{align}\label{prod-1}
\prod_{h=0}^{m-1}\Gamma_p\left(\frac{x+h}{m}\right)=\omega(m^{(1-x)(1-p)})~\Gamma_p(x)\prod_{h=1}^{m-1}\Gamma_p\left(\frac{h}{m}\right).
\end{align}
The next two relations are essentially contained in \cite{mccarthy-pacific}. Let $t\in\mathbb{Z}^{+}$ and $p\nmid t$. Then for 
$0\leq j\leq p-2$ we have
\begin{align}\label{gamma-prod-1}
\omega(t^{tj})\Gamma_p\left(\left\langle\frac{tj}{p-1}\right\rangle\right)\prod_{h=1}^{t-1}\Gamma_p\left(\frac{h}{t}\right)
=\prod_{h=0}^{t-1}\Gamma_p\left(\left\langle\frac{h}{t}+\frac{j}{p-1}\right\rangle\right),
\end{align}
and
\begin{align}\label{gamma-prod-2}
\omega(t^{-tj})\Gamma_p\left(\left\langle\frac{-tj}{p-1}\right\rangle\right)\prod_{h=1}^{t-1}\Gamma_p\left(\frac{h}{t}\right)
=\prod_{h=1}^{t}\Gamma_p\left(\left\langle\frac{h}{t}-\frac{j}{p-1}\right\rangle\right).
\end{align}
\begin{lemma}\label{new-lemma}
Let $m\geq1$ be positive integer and $p$ be an odd prime. For $1\leq j\leq p-2$ we have
$\left\lfloor\frac{mj}{p-1}\right\rfloor
=\displaystyle\sum_{h=0}^{m-1}\left\lfloor\frac{h}{m}+\frac{j}{p-1}\right\rfloor$,
$\left\lfloor\frac{1}{2}-\frac{mj}{p-1}\right\rfloor=\displaystyle\sum_{h=0}^{m-1}\left\lfloor\frac{1+2h}{2m}-\frac{j}{p-1}\right\rfloor$,
and
 $\left\lfloor\frac{-2j}{p-1}\right\rfloor=-1+\left\lfloor\frac{1}{2}-\frac{j}{p-1}\right\rfloor$.
\end{lemma}
\begin{proof}
As  $\left\lfloor\dfrac{mj}{p-1}\right\rfloor=0,1,\ldots,$ or $m-1$, it gives
$\left\lfloor\dfrac{mj}{p-1}\right\rfloor=\displaystyle\sum_{h=0}^{m-1}
\left\lfloor\frac{h}{m}+\frac{j}{p-1}\right\rfloor$.
Similarly, it is straightforward to verify the remaining equalities.
\end{proof}
\section{Proof of main theorems}
Before going to prove the main results regarding the values of ${_nG_n}(t)$ for arbitrary $t$ we prove two propositions.
\begin{proposition}\label{proposition-1}
Let $n\geq3$ be an integer and $p$ be an odd prime such that $p\nmid n$. For $t\in\mathbb{F}_p^{\times}$ let 
$B_n(t)=\displaystyle\sum_{\chi\in\widehat{\mathbb{F}_p^\times}}g(\varphi\chi^n)g(\overline{\chi}^n)\overline{\chi}((-1)^nt)$. If $t$ is $n$-th power residue modulo $p$ then we have 
$$B_n(t)=(p-1)g(\varphi)\displaystyle\sum_{\substack{a\in\mathbb{F}_p\\a^n\equiv t\pmod{p}}}\varphi(a)\varphi(a-1).$$ 
Otherwise, $B_n(t)=0$.

\end{proposition}
\begin{proof}
By \eqref{gauss-jacobi}, and \eqref{rel-2} it follows that
\begin{align}\label{eq-21}
B_n(t)&=\sum_{\chi\in\widehat{\mathbb{F}_p^\times}}g(\varphi\chi^n)g(\overline{\chi}^n)\overline{\chi}((-1)^nt)
=pg(\varphi)\sum_{\chi\in\widehat{\mathbb{F}_p^\times}}{\varphi\chi^n\choose\varphi}\overline{\chi}(t)\notag\\
&=\varphi(-1)g(\varphi)\sum_{y\in\mathbb{F}_p}\varphi(y)\varphi(1-y)\sum_{\chi\in\widehat{\mathbb{F}_p^\times}}\chi\left(\frac{y^n}{t}\right).
\end{align}
\eqref{orthogonal-1} gives that the latter sum present in \eqref{eq-21} is non zero only if $y^n\equiv t\pmod{p}$ has a solution in $\mathbb{F}_p^{\times}$. Therefore, if $t$ is $n$-th power residue modulo $p$ then we obtain $$B_n(t)=(p-1)g(\varphi)\displaystyle\sum_{\substack{a\in\mathbb{F}_p\\ a^n\equiv t\pmod{p}}}\varphi(a)\varphi(a-1).$$ On the other hand, if  $t$ is not a $n$-th power residue modulo $p$ then we have $B_{n}(t)=0$. This completes the proof of the proposition.
\end{proof}
\begin{proposition}\label{proposition-2}
Let $n\geq3$ be an integer and $p\geq 3$ be a prime such that $p\nmid n$. For $t\in\mathbb{F}_p^{\times}$ let
$B_n(t)=\displaystyle\sum_{\chi\in\widehat{\mathbb{F}_p^\times}}g(\varphi\chi^n)g(\overline{\chi}^n)\overline{\chi}((-1)^nt)$. Then
\begin{align}
B_n(t)=(p-1)g(\varphi)~
{_nG_n}(t).\notag
\end{align}
\end{proposition}
\begin{proof}
Taking $\chi=\omega^j$ and applying Gross-Koblitz formula we have 
\begin{align}\label{eq-1000}
B_n(t)&=\sum_{j=0}^{p-2}\pi^{(p-1)\ell_j}~\Gamma_p\left(\left\langle\frac{1}{2}-\frac{nj}{p-1}\right\rangle\right)
\Gamma_p\left(\left\langle\frac{nj}{p-1}\right\rangle\right)\overline{\omega}^j((-1)^nt),
\end{align}
where $\ell_j=\frac{1}{2}
-\left\lfloor\frac{1}{2}-\frac{nj}{p-1}\right\rfloor-\left\lfloor\frac{nj}{p-1}\right\rfloor$.
Applying \eqref{prod-1} (with $x=\left\langle\frac{1}{2}-\frac{nj}{p-1}\right\rangle$, and $m=n$) we obtain
\begin{align}\label{eq-1001}
\prod_{h=0}^{n-1}\Gamma_p\left(\left\langle\frac{1+2h}{2n}-\frac{j}{p-1}\right\rangle\right)
=\frac{\overline{\omega}^{nj}(n)}{\varphi(n)}\cdot
\Gamma_p\left(\left\langle\frac{1}{2}-\frac{nj}{p-1}\right\rangle\right)\prod_{h=1}^{n-1}\Gamma_p\left(\frac{h}{n}\right).
\end{align}
Also, \eqref{gamma-prod-1} yields
\begin{align}\label{eq-1002}
\omega^{nj}(n)\Gamma_p\left(\left\langle\frac{nj}{p-1}\right\rangle\right)\prod_{h=1}^{n-1}\Gamma_p\left(\frac{h}{n}\right)
=\prod_{h=0}^{n-1}\Gamma_p\left(\left\langle\frac{h}{n}+\frac{j}{p-1}\right\rangle\right).
\end{align}
By Lemma \ref{new-lemma} we have
\begin{align}\label{eq-1005}
\ell_j=\frac{1}{2}-\sum_{h=0}^{n-1}\left\{\left\lfloor\frac{1+2h}{2n}-\frac{j}{p-1}\right\rfloor+
\left\lfloor\frac{h}{n}+\frac{j}{p-1}\right\rfloor\right\}
\end{align}
Putting \eqref{eq-1001}, \eqref{eq-1002}, and \eqref{eq-1005} in \eqref{eq-1000} we obtain
\begin{align}\label{eq-1006}
\frac{B_n(t)}{1-p}=\pi^{\frac{(p-1)}{2}}\varphi(n)\prod_{h=0}^{n-1}\frac{\Gamma_p(\frac{1+2h}{2n})}{\Gamma_p(\frac{h}{n})}~{_n{G}_n}(t).
\end{align}
Using \eqref{eq-1001} with $j=0$ we obtain
$\displaystyle\prod_{h=0}^{n-1}\frac{\Gamma_p(\frac{1+2h}{2n})}{\Gamma_p(\frac{h}{n})}=\varphi(n)\Gamma_p\left(\frac{1}{2}\right)$.
Using this, and Gross-Koblitz formula in \eqref{eq-1006} we obtain
\begin{align*}
B_n(t)=(p-1)g(\varphi)~
{_nG_n}(t).
\end{align*}
This completes the proof.
\end{proof}
\begin{proof}[Proof of Theorem \ref{special-value-1}]
Let $B_n(t)=\displaystyle\sum_{\chi\in\widehat{\mathbb{F}_p^\times}}g(\varphi\chi^n)g(\overline{\chi}^n)\overline{\chi}((-1)^nt)$. Applying Proposition \ref{proposition-2} on the above sum we obtain
\begin{align}\label{proof-eq-1}
B_n(t)=(p-1)g(\varphi)~
{_nG_n}(t).
\end{align}
Let $t^\frac{(p-1)}{d}\equiv1\pmod{p}$. Then applying Proposition \ref{proposition-root} (with $b=t$) we obtain that 
$t$ is $n$-th power residue modulo $p$. Using this information and Proposition \ref{proposition-1} we obtain that
\begin{align}\label{proof-eq-2}
B_n(t)=(p-1)g(\varphi)\sum_{\substack{a\in\mathbb{F}_p\\ a^n\equiv t\pmod{p}}}\varphi(a)\varphi(a-1).
\end{align}
Combining \eqref{proof-eq-1}, and \eqref{proof-eq-2} we complete the proof of the first part. Similarly, we prove the second part.
\end{proof}
\begin{proof}[Proof of Corollary \ref{cor-8}]
By Proposition \ref{proposition-root}, and Theorem \ref{special-value-1} the result follows.
\end{proof}
\begin{proof}[Proof of Theorem \ref{analogue-1}]
If $p\equiv1\pmod{3}$ then we have $\gcd{(3,p-1)}=3$. Now,  for $t\in\mathbb{F}_p^{\times}-\{1\}$ let us assume that 
${_3G_3}(t)=0$. Since $g$ is a primitive root modulo ${p}$, so we write $t=g^i$ for some integer $i$. If $\gcd{(i,3)}=3$ then $i=3k$ for some integer $k$. This yields  
$$t^{\frac{p-1}{3}}=g^{\frac{3k(p-1)}{3}}\equiv1\pmod{p}.$$ If we use Theorem \ref{special-value-1} for $n=3$ then we obtain
\begin{align}\label{proof-eq-4}
{_3G_3}(t)=\sum_{\substack{a\in\mathbb{F}_p\\a^3\equiv t\pmod{p}}}\varphi(a(a-1)).
\end{align}
Since $y^3\equiv t\pmod{p}$ has 3 roots as $p\equiv1\pmod{3}$, so the R.H.S. of \eqref{proof-eq-4} cannot be equal to zero. However, this is a contradiction to the fact that ${_3G_3}(t)=0$. Therefore, we have $\gcd{(i,3)}=1$. Conversely, suppose that $\gcd{(i,3)}=1$. Then
$$t^{\frac{p-1}{3}}=g^{\frac{i(p-1)}{3}}\not\equiv1\pmod{p}.$$ Again, if we use Theorem \ref{special-value-1} for $n=3$ then we have
${_3G_3}(t)=0$.
Now, if $t=1$ then $y^3\equiv1\pmod{p}$ has three roots in $\mathbb{F}_p$. Let $a\neq1$ be a solution of $y^3\equiv1\pmod{p}$ in $\mathbb{F}_p^{\times}$. Then the complete list of solutions of this congruence are 1, $a$, $a^2=a^{-1}$, where $a^{-1}$ denotes the inverse of $a$ in $\mathbb{F}_p^\times$. Using this information and Theorem \ref{special-value-1} (with $n=3$, and $t=1$) we have
\begin{align}
{_3G_3}(1)=\varphi(a)\varphi(a-1)+\varphi(a^{-1})\varphi(a^{-1}-1).
\end{align}
Now, ${_3G_3}(1)=0$ if and only if $\varphi(a)\varphi(a-1)+\varphi(a^{-1})\varphi(a^{-1}-1)=0$. This is possible if and only if
$1+\varphi(-a)=0$. This is equivalent to $$1+\varphi(-1)\varphi(a^{-1})=1+\varphi(-1)\varphi(a^{2})=1+\varphi(-1)=0.$$ This is true if and only if $p\equiv3\pmod{4}$. As $p\equiv1\pmod{3}$ it follows that ${_3G_3}(1)=0$ if and only if $p\equiv7\pmod{12}$.\\
To prove the second part of the theorem let $p\not\equiv1\pmod{3}$. Then $\gcd{(3,p-1)}=1$, and for $t\in\mathbb{F}_p^{\times}$ it is well known that $t^{p-1}\equiv1\pmod{p}$. For $t\neq0,1$ applying Theorem \ref{special-value-1} (with $n=3$) and Proposition \ref{proposition-root} we obtain 
\begin{align}\label{proof-eq-5}
{_3G_3}(t)=\varphi(a(a-1)),
\end{align}
where $a^3\equiv t\pmod{p}$.
The R.H.S. of \eqref{proof-eq-5} can take only values $\pm1$. Thus if $t\neq0,1$ then ${_3G_3}(t)\neq0$. Now, if $t=1$ then the only solution of $y^3\equiv1\pmod{p}$ is 1. Using this information and Theorem \ref{special-value-1} we obtain
that ${_3G_3}(1)=0$. This completes the proof.
\end{proof}
\begin{proof}[Proof of Corollary \ref{cor-1}]
From the first and second parts of Theorem \ref{analogue-1} we  obtain that ${_3G_3}(1)=0$ if $p\equiv5,7,11\pmod{12}$. Now, if $p\equiv1\pmod{12}$ then the congruence $y^3\equiv1\pmod{p}$ has three solutions. Therefore, Theorem \ref{special-value-1} gives ${_3G_3}(1)$ is either equal to 2 or $-2$. This completes the proof.
\end{proof}
\begin{proof}[Proof of Corollary \ref{cor-10}]
The first part of the corollary follows easily from Theorem \ref{special-value-1}. Now, if $p\equiv2\pmod{3}$ then 
$\gcd{(3,p-1)=1}$ and $t^{p-1}\equiv1\pmod{p}$ for all $t$ such that $\gcd{(t,p)}=1$. Using this and Proposition \ref{proposition-root} we obtain that the congruence $y^3\equiv t\pmod{p}$ has a unique solution. To be specific it can be easily shown that the congruence has the unique solution $y\equiv t^{\frac{2p-1}{3}}\pmod{p}$. Using this fact and
Theorem \ref{special-value-1} we conclude the result.
\end{proof}
\begin{proof}[Proof of Corollary \ref{cor-2}]
By Theorem \ref{special-value-1} we have
\begin{align}\label{eq-main}
{_nG_n}(t)=\sum_{\substack{a\in\mathbb{F}_p\\
a^n\equiv t\pmod{p}}}\varphi(a)\varphi(a-1).
\end{align}
Now, if $n$ is even then $d$ is even. Let $t\neq1$. By Proposition \ref{proposition-root} we know that the congruence 
$y^n\equiv t\pmod{p}$ has $d$ solutions. If $a_1,a_2,\ldots,a_d$ are the modulo $p$ solutions of the congruence $y^n\equiv t\pmod{p}$ such that $\displaystyle\sum_{i=1}^{d}\varphi(a_i(a_i-1))=0$ then \eqref{eq-main} gives that ${_nG_n}(t)=0$. Let $t=1$ and $a_1,a_2,\ldots,a_{d-1}$ different from $1$ are the solutions of $y^n\equiv 1\pmod{p}$ modulo $p$. If possible let 
${_nG_n}(1)=0$. Then \eqref{eq-main} gives $\displaystyle\sum_{i=1}^{d-1}\varphi(a_i(a_i-1))=0$, which is not possible as $d$ is even, so ${_nG_n}(1)\neq0$.\\
To prove the second part let $n$ be odd. Then $d$ is odd. If $t\neq1$ then using Proposition \ref{proposition-root} we obtain that the congruence $y^n\equiv t\pmod{p}$ has $d$ solutions. If $a_1,a_2,\ldots,a_d$ are the solutions of the congruence 
$y^n\equiv t\pmod{p}$ then \eqref{eq-main} gives ${_nG_n}(t)=\displaystyle\sum_{i=1}^{d}\varphi(a_i(a_i-1)),$ which cannot be equal to zero as $d$ is odd. Similarly, we settle the case for $t=1$. 

\end{proof}
\begin{proof}[Proof of Corollary \ref{cor-6}]
If $\gcd{(n,p(p-1))}=1$ then it follows from Proposition \ref{proposition-root} that the congruence $$y^n\equiv t\pmod{p}$$ has a unique solution for each $t$ such that $\gcd{(t,p)}=1$. Now, if $t=1$ then 1 is the unique solution of the congruence $y^n\equiv 1\pmod{p}$. Using this information in Theorem \ref{special-value-1} we have ${_nG_n}(1)=0$. Let $t\neq1$ and 
$y\equiv a\pmod{p}$ be the unique solution of $y^n\equiv t\pmod{p}$. Then Theorem \ref{special-value-1} yields
${_nG_n}(t)=\varphi(a(a-1))$, which cannot be zero. This completes the proof.
\end{proof}
\begin{proof}[Proof of Corollary \ref{cor-7}]
If $n$ is even and $p\equiv3\pmod{4}$ then $\gcd{(n,p-1)}=2k$ for some odd integer $k$. This gives $(-1)^{\frac{p-1}{2k}}\equiv-1\pmod{p}$. Then it follows from Theorem \ref{special-value-1} that ${_nG_n}(-1)=0$. This completes the proof.
\end{proof}
\begin{proof}[Proof of Corollary \ref{cor-3}]
Let $n=p-1$ and $t\in\mathbb{F}_p-\{0,1\}$. By Theorem \ref{special-value-1} it follows that
${_{p-1}G_{p-1}(t)}=0$.

Now, we investigate the case when $t=1$. The fact that each element of $\mathbb{F}_p^{\times}$ satisfies the congruence $y^{p-1}\equiv1\pmod{p}$ is well known. Then Theorem \ref{special-value-1}, and \eqref{rel-1} yield
$
{_{p-1}G_{p-1}(1)}=\displaystyle\sum_{a\in\mathbb{F}_p^{\times}}\varphi(a)\varphi(a-1)=\varphi(-1)J(\varphi, \varphi)=-1$.
\end{proof}
We now provide two propositions. These propositions are used to examine the values of the function ${_n\widetilde{G}_n}(t)$.
\begin{proposition}\label{proposition-3}
Let $n\geq3$ be an integer and $p$ be a prime such that $p\nmid n(n-1)$. Let $t\in\mathbb{F}_p^{\times}$ and 
 $f_t(y)=y^{n}-y^{n-1}+\frac{(n-1)^{n-1}}{n^n}t\in\mathbb{F}_p[y]$ be a polynomial in $y$. Let $\alpha=\frac{4(1-n)^{n-1}}{n^n}$, and
$\beta_n(t)=\left\{\begin{array}{ll}
1, & \hbox{if $n$ is odd;}\\
1-(p-1)\varphi((1-n) t), & \hbox{if $n$ is even.}
\end{array}
\right.$

Let $A_n(t)=\displaystyle\sum_{\chi\in\widehat{\mathbb{F}_p^\times}}g(\varphi\chi^{n-1})g(\overline{\chi}^n)g(\overline{\chi})
g(\chi^2)\overline{\chi}(\alpha t)$. Then
\begin{align}
A_n(t)=p(p-1)g(\varphi)\displaystyle\sum_{\substack{a\in\mathbb{F}_p\\ f_t(a)\equiv0\pmod{p}}}\varphi(a(a-1))+(p-1)g(\varphi)\beta_n(t).\notag
\end{align}
\end{proposition}
\begin{proof}
Multiplying both numerator and denominator by $g(\varphi\overline{\chi})$ we have
\begin{align}\label{eq-1}
A_t&=\sum_{\chi\in\widehat{\mathbb{F}_p^\times}}\frac{g(\varphi\chi^{n-1})g(\overline{\chi}^n)}{g(\varphi\overline{\chi})}
g(\varphi\overline{\chi})g(\overline{\chi})g(\chi^2)\overline{\chi}(\alpha t).
\end{align}
Applying \eqref{hd} (with $m=2$) we have
$g(\varphi\overline{\chi})g(\overline{\chi})=g(\overline{\chi}^2)g(\varphi)\chi(4)$. Substituting this in \eqref{eq-1} and then using \eqref{gauss-jacobi}, \eqref{inverse}, and \eqref{rel-1} we obtain
\begin{align}
A_n(t)&=pg(\varphi)\sum_{\chi\in\widehat{\mathbb{F}_p^\times}}J(\varphi\chi^{n-1},\overline{\chi}^n)
\overline{\chi}\left(\frac{\alpha t}{4}\right)
+(p-1)g(\varphi)\beta_n(t).\notag
\end{align}
By \eqref{binomial}, and \eqref{rel-2} we re-write $A_n(t)$ as
\begin{align}\label{eq-4}
A_n(t)&=p^2g(\varphi)\sum_{\chi\in\widehat{\mathbb{F}_p^\times}}{\varphi\chi^{n-1}\choose\varphi\overline{\chi}}\overline{\chi}
\left(\frac{(-1)^{n}\alpha t}{4}\right)
+(p-1)g(\varphi)\beta_n(t)\notag\\
&=pg(\varphi)\sum_{1\neq y\in\mathbb{F}_p}\varphi\left(\frac{y}{y-1}\right)
\sum_{\chi\in\widehat{\mathbb{F}_p}^\times}\chi\left(\frac{4y^{n-1}(1-y)}{(-1)^{n-1}\alpha t}\right)+(p-1)g(\varphi)
\beta_n(t).
\end{align} 
By \eqref{orthogonal-1} we obtain that the second sum present on the R.H.S. of \eqref{eq-4} is non zero only if $y^n-y^{n-1}+\frac{(n-1)^{n-1}}{n^n}t\equiv0\pmod{p}$ admits a solution in $\mathbb{F}_p$. Using this information we have 
\begin{align}
A_n(t)=p(p-1)g(\varphi)\displaystyle\sum_{\substack{a\in\mathbb{F}_p\\ f_t(a)\equiv0\pmod{p}}}
\varphi(a(a-1))+(p-1)g(\varphi)\beta_n(t)\notag
\end{align}
\end{proof}
\begin{proposition}\label{proposition-4}
Let $n\geq3$ and $p\nmid n(n-1)$ be an odd prime. If $t\in\mathbb{F}_p^{\times}$ and
$\alpha=\frac{4(1-n)^{n-1}}{n^n}$ then let
 $A_n(t)=\displaystyle\sum_{\chi\in\widehat{\mathbb{F}_p^\times}}g(\varphi\chi^{n-1})g(\overline{\chi}^n)g(\overline{\chi})
g(\chi^2)\overline{\chi}\left(\alpha t\right).$
 Then we have
$A_n(t)=(p-1)g(\varphi)(1+p\cdot {_n\widetilde{G}_n}(t))$.
\end{proposition}
\begin{proof}
Replacing $\chi$ by $\omega^j$ and then applying Gross-Koblitz formula, \eqref{prod-1}, \eqref{gamma-prod-1}, and \eqref{gamma-prod-2} similarly as shown in the proof of Proposition \ref{proposition-2}  we deduce that 
\begin{align}\label{eq-11}
\frac{A_n(t)}{\varphi(n-1)}&=\sum_{j=1}^{p-2}\pi^{(p-1)\ell_j}~\overline{\omega}^j((-1)^{n-1}t)
\prod_{h=0}^{n-2}\frac{\Gamma_p\left(\left\langle\frac{1+2h}{2(n-1)}-\frac{j}{p-1}\right\rangle\right)}
{\Gamma_p(\frac{h}{n-1})}\\
&\times\frac{\Gamma_p\left(\left\langle\frac{1}{2}-\frac{j}{p-1}\right\rangle\right)}{\Gamma_p(\frac{1}{2})}
\prod_{h=0}^{n-1}\frac{\Gamma_p\left(\left\langle\frac{h}{n}+\frac{j}{p-1}\right\rangle\right)}{\Gamma_p(\frac{h}{n})}
\notag\\
&\times\Gamma_p\left(\left\langle\frac{j}{p-1}\right\rangle\right)\Gamma_p\left(\left\langle1-\frac{j}{p-1}\right\rangle\right)
+\pi^{\frac{p-1}{2}}\prod_{h=0}^{n-2}
\frac{\Gamma_p\left(\frac{1+2h}{2(n-1)}\right)}{\Gamma_p(\frac{h}{n-1})},\notag
\end{align}
where $\ell_j=\frac{1}{2}-\lfloor\frac{1}{2}-\frac{(n-1)j}{p-1}\rfloor-\lfloor\frac{nj}{p-1}\rfloor-\lfloor\frac{j}{p-1}\rfloor
-\lfloor\frac{-2j}{p-1}\rfloor$.
If $1\leq j\leq p-2$ then Gross-Koblitz formula, and \eqref{inverse} give
\begin{align}\label{eq-12}
\Gamma_p\left(\left\langle\frac{j}{p-1}\right\rangle\right)\Gamma_p\left(\left\langle1-\frac{j}{p-1}\right\rangle\right)
=-\omega^j(-1).
\end{align}
Also, Gross-Koblitz gives 
\begin{align}\label{last}
g(\varphi)=-\pi^{(p-1)/2}~\Gamma_p(1/2).
\end{align}
\eqref{prod-1} yields $\displaystyle\prod_{h=0}^{n-2}\frac{\Gamma_p(\frac{1+2h}{2(n-1)})}{\Gamma_p(\frac{h}{n-1})}=\varphi(n-1)
\Gamma_p(1/2)$. Substituting this identity along with \eqref{eq-12}, and \eqref{last} into \eqref{eq-11} and finally using Lemma \ref{new-lemma} in the expression of $\ell_j$ we deduce the required identity.
\end{proof}
\begin{proof}[Proof of Theorem \ref{special-value-2}]
Let
$A_n(t)=\displaystyle\sum_{\chi\in\widehat{\mathbb{F}_p^\times}}g(\varphi\chi^{n-1})g(\overline{\chi}^n)g(\overline{\chi})
g(\chi^2)\overline{\chi}\left(\alpha t\right),$ where\\  $\alpha=\frac{4(1-n)^{n-1}}{n^n}$. Now, applying 
Proposition \ref{proposition-3}, and Proposition \ref{proposition-4} on $A_n(t)$ and then combining both the expressions we obtain the result. 
\end{proof}

\begin{proof}[Proof of Corollary \ref{SV-2}]
Let $n=3$ and $f_t(y)=27y^3-27y^2+4t$. Then applying Theorem \ref{special-value-2} for $n=3$ and $p>3$ we obtain 
\begin{align}\label{eqn-1}
{_3\widetilde{G}_3}(t)=\sum_{\substack{a\in\mathbb{F}_p\\f_t(a)\equiv0\pmod{p}}}\varphi(a(a-1)).
\end{align}
For $p>3$ we know that if $t=1$ then 
the roots of the polynomial $27y^3-27y^2+4$ are $\frac{2}{3}$ with multiplicity two, and $\frac{-1}{3}$ with one.
Therefore, if $t=1$ then using this information in \eqref{eqn-1} we obtain 
\begin{align}
{_3\widetilde{G}_3}(1)=1+\varphi(-2).\notag
\end{align} 
This proves the first part. Similarly, we prove the other parts  of the corollary.
\end{proof}
\begin{proof}[Proof of Corollary \ref{cor-4}]
If $n$ is even and $t\in\mathbb{F}_p^\times$ then Theorem \ref{special-value-2} yields
\begin{align}\label{eq-1-cor-4}
{_n\widetilde{G}_n}(t)=\frac{(1-p)\varphi((1-n)t)}{p}+\sum_{\substack{a\in\mathbb{F}_p\\ f_t(a)\equiv0\pmod{p}}}
\varphi(a(a-1)).
\end{align}
If possible let ${_n\widetilde{G}_n}(t)=0$ then \eqref{eq-1-cor-4} gives
 $(p-1)\varphi(t)\varphi(1-n)\equiv0\pmod{p}$, which is not possible. Hence, if $t\neq0$ then ${_n\widetilde{G}_n}(t)\neq0$. This completes the proof.
\end{proof}


\end{document}